\documentclass[11pt]{article}
\usepackage{geometry}
\usepackage{color}
\usepackage{float}
\usepackage{textcomp}
\usepackage{url}
\usepackage{amsthm}
\usepackage{amsmath}
\usepackage{amssymb}%
\usepackage{mdframed}
\usepackage{multirow}
\usepackage{changes}

\usepackage{mathtools}
\usepackage{booktabs}    
\usepackage{subfigure} 
\usepackage{caption} 
\usepackage{multirow}
\usepackage{multicol}    
\usepackage{array}       
\usepackage{graphicx}
\graphicspath{{figuras/}}
\usepackage{empheq,amsmath}

\usepackage
[ 
lined, 
boxruled, 
commentsnumbered
]
{algorithm2e}
\DecMargin{0.1cm} 

\newtheorem{theorem}{Theorem}[section]
\newtheorem{lemma}[theorem]{Lemma}
\newtheorem{corollary}[theorem]{Corollary}

\newtheorem{remark}[theorem]{Remark}
\newtheorem{definition}[theorem]{Definition}

\newcommand{\bi}{\begin{itemize}}
\newcommand{\ei}{\end{itemize}}
\newcommand{\ba}{\begin{array}}
\newcommand{\ea}{\end{array}}

\begin{document}

\title{\textbf{Universal nonmonotone line search method for nonconvex multiobjective optimization problems with convex constraints}}


 \author{
 Maria Eduarda Pinheiro  \thanks{Trier University, Department of Mathematics, Universit\"{a}tsring 15, Trier, 54296, Trier, Germany. E-mail: pinheiro@uni-trier.de} \and  Geovani Nunes Grapiglia \thanks{Universit\'e catholique de Louvain, ICTEAM/INMA, Avenue Georges Lema\^itre, 4-6/ L4.05.01, B-1348, Louvain-la-Neuve, Belgium. E-mail: geovani.grapiglia@uclouvain.be} }

\date{November 14, 2024}

\maketitle

\begin{abstract}
In this work we propose a general nonmonotone line-search method for nonconvex multi\-objective optimization problems with convex constraints. At the $k$th iteration, the degree of nonmonotonicity is controlled by a vector $\nu_{k}$ with nonnegative components. Different choices for $\nu_{k}$ lead to different nonmonotone step-size rules. Assuming that the sequence $\left\{\nu_{k}\right\}_{k\geq 0}$ is summable, and that the $i$th objective function has Hölder continuous gradient with smoothness parameter $\theta_i \in(0,1]$, we show that the proposed method takes no more than $\mathcal{O}\left(\epsilon^{-\left(1+\frac{1}{\theta_{\min}}\right)}\right)$ iterations to find a $\epsilon$-approximate Pareto critical point for a problem with $m$ objectives and  $\theta_{\min}= \min_{i=1,\dots, m} \{\theta_i\}$. In particular, this complexity bound applies to the methods proposed by Drummond and Iusem (Comput. Optim. Appl. 28: 5--29, 2004), by Fazzio and Schuverdt (Optim. Lett. 13: 1365--1379, 2019), and by Mita, Fukuda and Yamashita (J. Glob. Optim. 75: 63--90, 2019).  The generality of our approach also allows the development of new methods for multiobjective optimization. As an example, we propose a new nonmonotone step-size rule inspired by the Metropolis criterion. Preliminary numerical results illustrate the benefit of nonmonotone line searches and suggest that our new rule is particularly suitable for multiobjective problems in which at least one of the objectives has many non-global local minimizers.
\\[2mm]
{\bf Keywords:} Multiobjective optimization; Projected gradient methods; Nonmonotone line searches; Worst-case complexity
\end{abstract}

\section{Introduction}
\subsection*{Motivation}
In this work, we consider multiobjective optimization problems with convex constraints. This type
of problem appears in several important applications, such as seismic design of
buildings \cite{32}, planetary exploration \cite{11,40}, maintenance of civil infrastructures \cite{17}, planning of cancer treatment \cite{45,12} and  
drug design \cite{39,31}. In most cases, there is not a single point which minimizes all the objective functions at once. This fact
motivates the notion of Pareto efficiency. Roughly speaking, a point is
said to be a Pareto efficient solution when from this point it is impossible to obtain
an improvement in any of the objective functions without worsening the value of some
other objective.

There are many approaches to solve multiobjective optimization problems \cite{36,33}. In recent years, several methods have been obtained by extending well-known optimization algorithms for single-objective optimization. Notable 
examples include the steepest descent method proposed in \cite{Fliege}, projected
gradient methods \cite{Drummond,18,19,5}, the Newton's method proposed in \cite{Fliege2}, proximal-point methods \cite{4,1}, trust-region methods \cite{Villacorta,9}, and also nonmonotone line search methods \cite{FS,MITA}. 

In the present work, we extend the general nonmonotone method proposed in \cite{ss} for single-objective optimization to multiobjective optimization. The method is said to be nonmonotone because it permits an increase in the objectives between consecutive iterations. At the $k$th iteration, the extent of the increase allowed is governed by a vector $\nu_{k}$ with nonnegative components. Various selections for $\nu_{k}$ result in different nonmonotone step-size rules. Specifically, considering suitable choices for $\nu_{k}$, our method encompasses instances of the methods proposed in \cite{Drummond,FS,MITA}. Assuming that the sequence $\left\{\nu_{k}\right\}_{k\geq 0}$ is summable, and that the $i$th objective function has Hölder continuous gradient with constant $H_{i}$ and smoothness parameter $\theta_i \in(0,1]$, we show that the proposed method takes no more than $\mathcal{O}\left(\epsilon^{-\left(1+\frac{1}{\theta_{\min}}\right)}\right)$ iterations to find a $\epsilon$-approximate Pareto critical point of a problem with $m$ objectives and  $\theta_{\min}= \min_{i=1,\dots, m} \{\theta_i\}$. For the case $m=1$, this bound agrees in order with the bound proved in \cite{Maryam} for a gradient method that requires the knowledge of $H_{1}$ and $\theta_{1}$. Since our method is fully adaptive with respect to the H\"{o}lder constants, we say that it is a \textit{universal method} in the sense of Nesterov \cite{Nesterov}. For the case in which $\theta_{1}=\ldots=\theta_{m}=\theta$, i.e., all the objective functions have H\"{o}lder continuous gradients \textit{with the same smoothness level $\theta$}, our complexity bound agrees in order with the bound established by \cite{mar} for the first-order version of their $p$th-order method. In addition, when all the objectives have Lipschitz continuous gradients (case $\theta_{1}=\ldots=\theta_{m}=1$), our bound becomes of $\mathcal{O}\left(\epsilon^{-2}\right)$ which agrees in order with the bounds established in \cite{24} for a multiobjective trust-region method, and in \cite{Fliege3} for the multiobjective steepest descent method. Under the weaker assumption that $\left\{\nu_{k}\right\}_{k\geq 0}$ converges to zero, we also prove a liminf-type global convergence result for our method.  The generality of our results regarding possible choices of the sequence $\left\{\nu_{k}\right\}_{k\geq 0}$ enables the development of new nonmonotone methods with convergence and worst-case complexity guarantees for multiobjective optimization. As an example, we propose a new Metropolis-based nonmonotone step-size rule inspired by a method recently proposed in \cite{GeovaniSachs} for single-objective optimization.

\subsection*{Contents}
This paper is organized as follows. In Section \ref{sec1}, we define the problem and review some results about multiobjective optimization. In Section \ref{sec2} we present our general nonmonotone method and establish its complexity and convergence properties. In Section \ref{sec3}, we analyse some particular instances of the method. In Section \ref{sec4} we propose a new nonmonotone method that fits into our general scheme. Finally, in Section \ref{sec5}, we report some illustrative numerical results.

\subsection*{Notations}

In what follows, $\|\,.\,\|$ denotes the Euclidean norm, $I=\left\{1,2,\ldots,m\right\}$, $\mathbb{R}_{+}^{m}~=~\left\{z\in\mathbb{R}^{m}\,|\,z_{i}\geq 0,\,i\in I\right\}$, $\mathbb{R}_{++}^{m}=\left\{z\in\mathbb{R}^{m}\,|\,z_{i}>0,\,i\in I\right\}$ and, given $x\in \Omega$, we consider the set $\Omega-x=\left\{y\in\mathbb{R}^{n}\,|\,y+x\in\Omega\right\}$. In addition, $J_{F}(x)$ will denote the Jacobian of $F:\mathbb{R}^{n}\to\mathbb{R}^{m}$ at $x$. The relations $\succ$ and $\succ_{\omega}$ are given respectively by
\begin{equation*}
y\succ x\Longleftrightarrow y-x\in\mathbb{R}_{+}^{m}\,\,\text{and}\,\,y\succ_{\omega}x\Longleftrightarrow y-x\in\mathbb{R}_{++}^{m}.
\end{equation*}
Moreover, given a symmetric matrix $B\in\mathbb{R}^{n\times n}$, $\lambda_{\min}(B)$ and $\lambda_{\max}(B)$ will denote the smallest and the largest eigenvalues of $B$. Finally, given a finite set $A$, $|A|$ denotes the cardinality of $A$.

\section{Problem Definition and Auxiliary Results}\label{sec1}

In this paper we consider methods for solving the following multiobjective optimization problem
\begin{eqnarray}
\min& F(x)=(f_{1}(x),\ldots,f_{m}(x))^{T},\label{eq:2.1}\\
\text{s.t.}& x\in \Omega,\label{eq:2.2}
\end{eqnarray}
where $\emptyset\neq\Omega\subset\mathbb{R}^{n}$ is a closed and convex set and $F:\mathbb{R}^{n}\to\mathbb{R}^{m}$ is a continuously differentiable function, possibly nonconvex. Let us start by recalling the definitions of efficient and (local) weakly efficient solutions of (\ref{eq:2.1})-(\ref{eq:2.2}).
\vspace{0.2cm}
\begin{definition}
    [\cite{Guerraggio}, page 619] Given a point $x^{*}\in\Omega$,
\begin{itemize}
\item[(a)] $x^{*}$ is said to be a Pareto efficient solution of (\ref{eq:2.1})-(\ref{eq:2.2}) when there is no $y\in\Omega$ such that $F(x^{*})\succ F(y)$ and $F(y)\neq F(x^{*})$;
\item[(b)] $x^{*}$ is said to be a weakly Pareto efficient solution of (\ref{eq:2.1})-(\ref{eq:2.2}) when there is no $y\in\Omega$ such that $F(x^{*})\succ_{w} F(y)$; and
\item[(c)] $x^{*}$ is said to be a local (or local weakly) Pareto efficient solution of (\ref{eq:2.1})-(\ref{eq:2.2}) when there exists a neighborhood $N(x^{*})$ of $x^{*}$ for which there is no $y\in N(x^{*})\cap\Omega$ such that $F(x^{*})\succ F(y)$ and $F(y)\neq F(x^{*})$ (or, respectively, $F(x^{*})\succ_{w} F(y)$).
\end{itemize}
\end{definition}
\noindent The lemma below gives a necessary condition for a point $x^{*}\in\Omega$ to be a local weakly efficient solution of (\ref{eq:2.1})-(\ref{eq:2.2}).
\vspace{0.2cm}
\begin{lemma}
\label{lem:extra1}
Let $F:\mathbb{R}^{n}\to\mathbb{R}^{m}$ be a continuously differentiable function and $\Omega\subset\mathbb{R}^{n}$ be a closed and convex set. If $x^{*}\in\Omega$ is a local weakly efficient solution of (\ref{eq:2.1})-(\ref{eq:2.2}), then
\begin{equation}
-\mathbb{R}_{++}^{m}\cap\left\{J_{F}(x^{*})(x-x^{*})\,|\,x\in\Omega\right\}=\emptyset.
\label{eq:2.3}
\end{equation}
\end{lemma}
\begin{proof}
See Theorem 5.1 (item (ii)-(a)) in  \cite{Guerraggio}.
\end{proof}
\noindent Lemma \ref{lem:extra1} motivates the following definition.
\vspace{0.2cm}
\begin{definition} A point $x^{*}\in\Omega$ is said to be a Pareto critical point for (\ref{eq:2.1})-(\ref{eq:2.2}) if it satisfies condition (\ref{eq:2.3}).
\end{definition}
\vspace{0.2cm}
\begin{remark}
Note that when $m=1$ and $\Omega=\mathbb{R}^{n}$, (\ref{eq:2.1})-(\ref{eq:2.2}) reduces to the unconstrained scalar optimization problem, and the Pareto criticality condition (\ref{eq:2.3}) is equivalent to the classical stationarity condition  $\nabla f(x)=0$.
\end{remark}
\vspace{0.2cm}
Let $x\in\Omega$ be a point which is not Pareto critical. Then, there exists a direction $d\in\Omega-x$ such that
\begin{equation*}
J_{F}(x)d\in -\mathbb{R}_{++}^{m},
\end{equation*}
that is,
\begin{equation*}
J_{F}(x)d\prec_{w} 0.
\end{equation*}
In this case, $d$ is called a descent direction for $F$ at $x$. According to the next lemma, a descent direction of $F$ at $x$ can be obtained by solving the following problem:
\begin{equation}
\min_{d\in \Omega-x}h_{x}(d)\equiv \max_{i\in I}\left\{\nabla f_{i}(x)^{T}d\right\}+\dfrac{\|d\|^{2}}{2}.
\label{eq:2.4}
\end{equation}
\vspace{0.2cm}
\begin{lemma}
\label{lem:2.1}
The following statements hold:
\begin{itemize}
\item[(a)] The subproblem (\ref{eq:2.4}) has only one solution.
\item[(b)] If $x$ is a Pareto critical point of $F$ and $s(x)=\arg\min_{d\in \Omega-x} h_{x}(d)$, then $s(x)=0$ and consequently $h_{x}(s(x))=0$.
\item[(c)] If $x\in\Omega$ is not a Pareto critical point of $F$, then $s(x)\neq 0$ and $h_{x}(s(x))<0$. In particular, $s(x)$ is a descent direction for $F$ at $x$.
\item[(d)] The mapping $x\longmapsto s(x)$ is continuous.
\end{itemize}
\end{lemma}

\begin{proof}
See Proposition 3 in \cite{Drummond} and Lemma 1 in \cite{Fliege}. 
\end{proof}
\vspace{0.2cm}
By statements (b) and (d) in Lemma \ref{lem:2.1}, $\|s(x)\|$ is a suitable Pareto criticality measure for $x$. This remark motivates the following definition.

\begin{definition}
Given $\epsilon$, we say that $x$ is an $\epsilon$-approximate Pareto critical point for (\ref{eq:2.1})-(\ref{eq:2.2}) when $\|s(x)\|\leq\epsilon$.
\end{definition}
\vspace{0.2cm}
It is worth mentioning that this is not the only way to define $\epsilon$-approximate Pareto criticality. For example, in the case $\Omega=\mathbb{R}^{n}$, \cite{Cocchi} consider $x$ as an $\epsilon$-approximate Pareto critical point when
\begin{equation*}
\min_{d\in B[0;1]}\max_{i\in I}\left\{\nabla f_{i}(x)^{T}d\right\}\geq -\epsilon
\end{equation*}
or, equivalently,
\begin{equation}
-\min_{d\in B[0;1]}\max_{i\in I}\left\{\nabla f_{i}(x)^{T}d\right\}\leq \epsilon,
\label{eq:pareto2}    
\end{equation}
where $B[0,1]=\left\{x\in\mathbb{R}^{n}\,:\,\|x\|\leq 1\right\}$. Notice that 
\begin{eqnarray}
\min_{d\in B[0;1]}\max_{i\in I}\left\{\nabla f_{i}(x)^{T}d\right\}&\leq &\max_{i\in I}\left\{\nabla f_{i}(x)^{T}\left(\frac{s(x)}{\|s(x)\|}\right)\right\}\nonumber\\
&  = & \dfrac{1}{\|s(x)\|}\max_{i\in I}\left\{\nabla f_{i}(x)^{T}s(x)\right\}\nonumber\\
&\leq &\dfrac{1}{\|s(x)\|}\left(\max_{i\in I}\left\{\nabla f_{i}(x)^{T}s(x)\right\}+\frac{1}{2}\|s(x)\|^{2}\right).
\label{eq:A}
\end{eqnarray}
Denote 
\begin{equation}
    \xi(x)=\max_{i\in I}\left\{\nabla f_{i}(x)^{T}s(x)\right\}+\frac{1}{2}\|s(x)\|^{2}.
    \label{eq:C}
\end{equation}
By Proposition 2 in \cite{MITA}, we have
\begin{equation}
\|s(x)\|^{2}=2(-\xi(x)).
    \label{eq:B}
\end{equation}
Combining (\ref{eq:A}), (\ref{eq:C}) and (\ref{eq:B}), it follows that
\begin{equation*}
\min_{d\in B[0;1]}\max_{i\in I}\left\{\nabla f_{i}(x)^{T}d\right\}\leq -\dfrac{\|s(x)\|}{2}.
\end{equation*}
Consequently,
\begin{equation*}
\|s(x)\|\leq 2\left(-\min_{d\in B[0;1]}\max_{i\in I}\left\{\nabla f_{i}(x)^{T}d\right\}\right).
\end{equation*}
Therefore, if $x$ is an $\epsilon$-approximate Pareto critical point in the sense of (\ref{eq:pareto2}), then $\|s(x)\|\leq 2\epsilon$.

\section{Universal Nonmonotone Line Search Methods}
\label{sec2}
Let us consider the following general algorithm to solve (\ref{eq:2.1})-(\ref{eq:2.2}):
\begin{algorithm}
\caption{Universal Nonmonotone Line Search Method}\label{algo1}
\noindent\textbf{Step 0.} Choose $x_{0}\in\Omega$, $m_{0}\in\left\{0,1,\ldots,m\right\}$, $c_{1},c_{2}>0$, and $\beta,\rho\in (0,1)$. Set $k:=0$.
\\[0.2cm]
\noindent\textbf{Step 1.} Compute $d(x_{k})\in\Omega-x_{k}$ such that
\begin{equation}
\max_{i\in I}\left\{\nabla f_{i}(x_{k})^{T}d(x_{k})\right\}\leq -c_{1}\|s(x_{k})\|^{2}
\label{eq:3.1}
\end{equation} 
and
\begin{equation}
\|d(x_{k})\|\leq c_{2}\|s(x_{k})\|.
\label{eq:3.2}
\end{equation}
\vspace{0.2cm}
\noindent\textbf{Step 2.1.} Set $\ell:=0$.
\\[0.2cm]
\noindent\textbf{Step 2.2.} Choose $\nu_{k,\ell}\in\mathbb{R}^{m}_{+}$. If 
\begin{equation}
\left|\left\{i\in I\,:\,f_{i}(x_{k}+\beta^{\ell}d(x_{k}))\leq f_{i}(x_{k})+\rho\beta^{\ell}\nabla f_{i}(x_{k})^{T}d(x_{k})\right\}\right|\geq m_{k},
\label{eq:cardinality}
\end{equation}
and
\begin{equation}
f_{i}(x_{k}+\beta^{\ell}d(x_{k}))\leq f_{i}(x_{k})+\rho\beta^{\ell}\nabla f_{i}(x_{k})^{T}d(x_{k})+\left[\nu_{k,\ell}\right]_{i}\quad\forall i\in I,
\label{eq:3.3}
\end{equation}
set $\ell_{k}=\ell$, and go to Step 3. Otherwise, set $\ell:=\ell+1$ and go back to Step 2.2.
\\[0.2cm]
\noindent\textbf{Step 3.} Define $\alpha_{k}=\beta^{\ell_{k}}$, $x_{k+1}=x_{k}+\alpha_{k}d(x_{k})$, and $\nu_{k}=\nu_{k,\ell_{k}}$. Choose $m_{k+1}\in\left\{0,1,\ldots,m\right\}$, set $k:=k+1$ and go back to Step 1.
\end{algorithm}
\vspace{0.2cm}

\begin{remark}
A natural choice for the search direction is $d(x_{k})=s(x_{k})$, which satisfies conditions (\ref{eq:3.1})-(\ref{eq:3.2}) with $c_{1}=c_{2}=1$. In the case $\Omega=\mathbb{R}^{n}$, given symmetric positive definite matrices $B_{i}(x_{k})\in\mathbb{R}^{n\times n}$ ($i=1,\ldots,m$), we can also use the search direction
\begin{equation}
d(x_{k})=\arg\min_{d\in\mathbb{R}^{n}}\max_{i\in I}\left\{\nabla f_{i}(x_{k})^{T}d+\frac{1}{2}d^{T}B_{i}(x_{k})d\right\}.
\label{eq:QN}
\end{equation}
Indeed, if $\lambda_{\max}\left(B_{i}(x_{k})\right)\leq 1/2c_{1}$, and $\lambda_{\min}\left(B_{i}(x_{k})\right)\geq 1/c_{2}$, then the vector $d(x_{k})$ given in (\ref{eq:QN}) also satisfies (\ref{eq:3.1})-(\ref{eq:3.2})\footnote{The proof for this fact follows exactly as in the proof of Proposition 2 in \cite{MITA}, replacing there $x$ with $x_{k}$, $\nabla^{2}f_{i}(x)$ with $B_{i}(x_{k})$, $\zeta$ with $1/2c_{1}$, $\xi$ with $1/c_{2}$; then using the equality $|\xi(x_{k})|=\|s(x_{k})\|/2$, where $\xi(x_{k})$ is defined in (\ref{eq:C}).}. In particular, when all objectives are strongly convex with Lipschitz continuous gradients, the Newtonian direction will satisfy (\ref{eq:3.1})-(\ref{eq:3.2}) for constants $c_{1}$ and $c_{2}$ that depend on the extreme eigenvalues of the Hessian matrices of the objectives.
\end{remark}
\vspace{0.2cm}
\begin{remark}
When $[\nu_{k,\ell}]_{i}>0$, condition (12) permits the acceptance of a stepsize $\beta^{\ell}$ even if $f_{i}(x_{k}+\beta^{\ell}d(x_{k}))>f_{i}(x_{k})$. In contrast, when $m_{k}\geq 1$, condition (11) requires a monotonic decrease for at least $m_{k}$ of the objectives.
\end{remark}

\vspace{0.2cm}
\noindent The analysis of  Algorithm~\ref{algo1} will be done under the following assumptions.
\\[0.2cm]
\noindent\textbf{A1} For each $i=1,\ldots,m$, the objective function $f_{i}$ belongs to the class of functions $C^{1,\theta_{i}}_{H_{i}}\left(\Omega\right)$, $\theta_{i}\in (0,1]$, which have H\"{o}lder continuous gradients:
\begin{equation*}
\|\nabla f_{i}(x)-\nabla f_{i}(y)\|\leq H_{i}\|y-x\|^{\theta_{i}},\,\,\text{for all}\,\,x,y\in \Omega.
\end{equation*}
\noindent\textbf{A2} For each $i\in I$, there exists $f_{i}^{*}\in\mathbb{R}$ such that $f_{i}(x)\geq f_{i}^{*}$ for all $x\in\Omega$.
\\[0.2cm]
\noindent\textbf{A3} For each $i\in I$, $\lim_{T\to +\infty}\frac{1}{T}\sum_{k=0}^{T-1}\left[\nu_{k}\right]_{i}=0$.
\\[0.2cm]
\noindent The next lemma establishes that   Algorithm~\ref{algo1} is well-defined.
\vspace{0.2cm}
\begin{lemma}

\label{lem:3.1}
Suppose that A1 holds. If $s(x_{k})\neq 0$, then there exists $\ell\in\mathbb{N}$ such that 
\begin{equation}
f_{i}(x_{k}+\beta^{\ell}d(x_{k}))\leq f_{i}(x_{k})+\rho\beta^{\ell}\nabla f_{i}(x_{k})^{T}d(x_{k})
\label{eq:monotone}
\end{equation}
for all $i\in I$.
\end{lemma}

\begin{proof}
Since $s(x_{k})\neq 0$, it follows from (\ref{eq:3.1}) that
\begin{equation*}
\max_{i\in I}\left\{\nabla f_{i}(x_{k})^{T}d(x_{k})\right\}<-c_{1}\|s(x_{k})\|^{2}<0.
\end{equation*}
Thus, for any $i\in I$, we have
\begin{equation}
\nabla f_{i}(x_{k})^{T}d(x_{k})<0.
\label{eq:3.4}
\end{equation}
In view of A1, $f_{i}$ is differentiable, and so
\begin{equation}
\lim_{\alpha\to 0^{+}}\dfrac{f_{i}(x_{k}+\alpha d(x_{k}))-f_{i}(x_{k})}{\alpha}=\nabla f_{i}(x_{k})^{T}d(x_{k})<\rho\nabla f_{i}(x_{k})^{T}d(x_{k}),
\label{eq:3.5}
\end{equation}
where the last inequality is due to (\ref{eq:3.4}) and $\rho\in (0,1)$. As a consequence of (\ref{eq:3.5}), there exists $\delta_{i}>0$ such that
\begin{equation}
f_{i}(x_{k}+\alpha d(x_{k}))\leq f_{i}(x_{k})+\rho\alpha\nabla f_{i}(x_{k})^{T}d(x_{k})
\label{eq:3.6}
\end{equation}
for all $\alpha\in (0,\delta_{i}]$. Thus, defining $\delta=\min_{i\in I}\left\{\delta_{i}\right\}$, it follows that (\ref{eq:3.6}) holds for all $i\in I$ as long as $\alpha\in (0,\delta]$. Since $\beta\in (0,1)$, there exists $\ell\in\mathbb{N}$ such that $\beta^{\ell_{k}}\leq\delta$. Therefore, for such $\ell$, (\ref{eq:monotone}) holds for all $i\in I$.
\end{proof}
\noindent The following lemma gives a lower bound for the sequence $\left\{\alpha_{k}\right\}$.
\vspace{0.2cm}
\begin{lemma}
\label{lem:lower}
Suppose that A1 hold and let $\left\{x_{k}\right\}_{k=0}^{T}$ be generated by  Algorithm~\ref{algo1} with 
\begin{equation}
\|s(x_{k})\|>\epsilon,\quad\text{for}\,\,k=0,\ldots,T-1,
\label{eq:extra2}
\end{equation}
for some $\epsilon\in (0,1)$.
Then
\begin{equation}
\alpha_{k}\geq\kappa_{1}\epsilon^{\frac{1-\theta_{\min}}{\theta_{\min}}}
\label{eq:3.7}
\end{equation}
for $k=0,\ldots,T-1$, where
\begin{equation}
\kappa_{1}=\min\left\{1,\min_{i\in I}\left\{\beta\left[\dfrac{(1+\theta_{i})c_{1}(1-\rho)}{c_{2}^{1+\theta_{i}}H_{i}}\right]^{\frac{1}{\theta_{i}}}\right\}\right\}
\label{eq:3.8}
\end{equation}
and
\begin{equation}
\theta_{\min}=\min_{i\in I}\left\{\theta_{i}\right\}.
\label{eq:3.9}
\end{equation}
\end{lemma}

\begin{proof}
Consider $k\in\left\{0,\ldots,T-1\right\}$ and consider the index
\begin{equation*}
\hat{\ell}_{k}\equiv\min\left\{\ell\in\mathbb{N}\,:\,\text{(\ref{eq:monotone}) holds for all $i\in I$}\right\},
\end{equation*}
which is well defined due to (\ref{eq:extra2}) and Lemma \ref{lem:3.1}. We will show first that
\begin{equation}
\beta^{\hat{\ell}_{k}}\geq\kappa_{1}\epsilon^{\frac{1-\theta_{\min}}{\theta_{\min}}}.
\label{eq:beta}
\end{equation}
If $\hat{\ell}_{k}=0$, then
\begin{equation*}
\beta^{\hat{\ell}_{k}}=1>\epsilon^{\frac{1-\theta_{\min}}{\theta_{\min}}}\geq\kappa_{1}\epsilon^{\frac{1-\theta_{\min}}{\theta_{\min}}},
\end{equation*}
that is, (\ref{eq:beta}) holds. Now, suppose that $\hat{\ell}_{k}>0$. Then, by the definition of $\hat{\ell}_{k}$, there exists $i\in I$ such that
\begin{equation}
f_{i}\left(x_{k}+\beta^{\hat{\ell}_{k}-1}d(x_{k}))-f_{i}(x_{k}\right)>\rho\beta^{\hat{\ell}_{k}-1}\nabla f_{i}(x_{k})^{T}d(x_{k})+\left[\nu_{k,\ell}\right]_{i}.
\label{eq:3.10}
\end{equation}
On the other hand, by A1 and Lemma 1 in \cite{Maryam} we have
\begin{equation}
f_{i}\left(x_{k}+\beta^{\hat{\ell}_{k}-1}d(x_{k})\right)\leq f_{i}(x_{k})+\beta^{\hat{\ell}_{k}-1}\nabla f_{i}(x_{k})^{T}d(x_{k})+\dfrac{H_{i}\left(\beta^{\hat{\ell}_{k}}-1\right)^{1+\theta_{i}}}{1+\theta_{i}}\|d(x_{k})\|^{1+\theta_{i}}.
\label{eq:3.11}
\end{equation}
Combining (\ref{eq:3.10}) and (\ref{eq:3.11}), it follows that
\begin{equation*}
\rho\beta^{\hat{\ell}_{k}-1}\nabla f_{i}(x_{k})^{T}d(x_{k})<\beta^{\hat{\ell}_{k}-1}\nabla f_{i}(x_{k})^{T}d(x_{k})+\dfrac{H_{i}\left(\beta^{\hat{\ell}_{k}-1}\right)^{1+\theta_{i}}}{1+\theta_{i}}\|d(x_{k})\|^{1+\theta_{i}},
\end{equation*}
which implies that
\begin{equation}
\left(\beta^{\hat{\ell}_{k}-1}\right)^{\theta_{i}}>\dfrac{(1+\theta_{i})(1-\rho)}{H_{i}}\left(-\dfrac{\nabla f_{i}(x_{k})^{T}d(x_{k})}{\|d(x_{k})\|^{1+\theta_{i}}}\right).
\label{eq:3.12}
\end{equation}
By (\ref{eq:3.1}) and (\ref{eq:3.2}) we have
\begin{equation}
-\dfrac{\nabla f_{i}(x_{k})^{T}d(x_{k})}{\|d(x_{k})\|^{1+\theta_{i}}}\geq\dfrac{c_{1}\|s(x_{k})\|^{2}}{c_{2}^{1+\theta_{i}}\|s(x_{k})\|^{1+\theta_{i}}}=\dfrac{c_{1}}{c_{2}^{1+\theta_{i}}}\|s(x_{k})\|^{1-\theta_{i}}.
\label{eq:3.13}
\end{equation}
Combining (\ref{eq:3.12}) and (\ref{eq:3.13}), we obtain
\begin{equation*}
\beta^{\hat{\ell}_{k}}=\beta\left(\beta^{\hat{\ell}_{k}-1}\right)>\beta\left[\dfrac{(1+\theta_{i})c_{1}(1-\rho)}{c_{2}^{1+\theta_{i}}H_{i}}\right]^{\frac{1}{\theta_{i}}}\|s(x_{k})\|^{\frac{1-\theta_{i}}{\theta_{i}}}.
\end{equation*}
Then, by (\ref{eq:extra2}) we have
\begin{equation*}
\beta^{\hat{\ell}_{k}}>\beta\left[\dfrac{(1+\theta_{i})c_{1}(1-\rho)}{c_{2}^{1+\theta_{i}}H_{i}}\right]^{\frac{1}{\theta_{i}}}\epsilon^{\frac{1-\theta_{i}}{\theta_{i}}}\geq \min_{i\in I}\left\{\beta\left[\dfrac{(1+\theta_{i})c_{1}(1-\rho)}{c_{2}^{1+\theta_{i}}H_{i}}\right]^{\frac{1}{\theta_{i}}}\right\}\epsilon^{\frac{1-\theta_{\min}}{\theta_{\min}}}\geq\kappa_{1}\epsilon^{\frac{1-\theta_{\min}}{\theta_{\min}}},
\end{equation*}
that is, (\ref{eq:beta}) also holds when $\hat{\ell}_{k}>0$. Finally, since $\ell_{k}\leq\hat{\ell}_{k}$,  it follows from (\ref{eq:beta}) that
\begin{equation*}
    \alpha_{k}=\beta^{\ell_{k}}\geq\beta^{\hat{\ell_{k}}}\geq\kappa_{1}\epsilon^{\frac{1-\theta_{\min}}{\theta_{\min}}}.
\end{equation*}
\end{proof}

\begin{remark}
By Lemma \ref{lem:lower} and the fact that $\alpha_{k}=\beta^{\ell_{k}}$, it follows that
\begin{equation*}
\ell_{k}\leq\dfrac{\log\left(\kappa_{1}^{-1}\epsilon^{-\left(\frac{1-\theta_{\min}}{\theta_{\min}}\right)}\right)}{|\log(\beta)|}.
\end{equation*}
This means that, each iteration of Algorithm 1 with $\|s(x_{k})\|>\epsilon$ requires the computation of one Jacobian matrix of $F(\,\cdot\,)$ and at most $\mathcal{O}\left(\log\left(\epsilon^{-\left(\frac{1-\theta_{\min}}{\theta_{\min}}\right)}\right)\right)$ evaluations of $F(\,\cdot\,)$.
\end{remark}
\vspace{0.2cm}
Given $i\in I$, it follows from Assumption A3 that for any $\delta>0$ there exists $C_{i}(\delta)>0$ such that
\begin{equation}
\dfrac{1}{T}\sum_{k=0}^{T-1}[\nu_{k}]_{i}\leq\delta,\quad\forall T\geq C_{i}(\delta).
\label{eq:3.14}
\end{equation}
The theorem below establishes an upper bound for the number of iterations that  Algorithm~\ref{algo1} need to find an $\epsilon$-approximate Pareto critical point. The proof is a direct adaptation of the proof of Theorem 2 in \cite{GeovaniSachs}.
\vspace{0.2cm}
\begin{theorem}
\label{thm:3.1}
Suppose that A1-A3 hold and let $\left\{x_{k}\right\}_{k=0}^{T}$ be generated by  Algorithm~\ref{algo1} with 
\begin{equation}
\|s(x_{k})\|>\epsilon,\quad\text{for}\,\,k=0,\ldots,T-1,
\label{eq:3.15}
\end{equation}
for some $\epsilon\in (0,1)$. Then
\begin{equation}
T\leq\min_{i\in I}\max\left\{C_{i}\left(\dfrac{\kappa_{2}}{2}\epsilon^{\left(1+\frac{1}{\theta_{\min}}\right)}\right),\dfrac{2(f_{i}(x_{0})-f_{i}^{*})}{\kappa_{2}}\epsilon^{-\left(1+\frac{1}{\theta_{\min}}\right)}\right\}
\label{eq:3.16}
\end{equation}
where $C_{i}(\,\cdot\,)$ is defined in (\ref{eq:3.14}) and 
\begin{equation}
\kappa_{2}=c_{1}\rho\kappa_{1},
\label{eq:3.17}
\end{equation}
with $\kappa_{1}$ given in (\ref{eq:3.8}).
\end{theorem}

\begin{proof}
Let $i\in I$. By (\ref{eq:3.3}), (\ref{eq:3.13}), (\ref{eq:3.15}), Lemma \ref{lem:lower}, and (\ref{eq:3.17}), we have
\begin{eqnarray*}
\left[\nu_{k}\right]_{i}+f_{i}(x_{k})-f_{i}(x_{k+1})&\geq &\rho\alpha_{k}\left(-\nabla f_{i}(x_{k})^{T}d(x_{k})\right)\geq c_{1}\rho\alpha_{k}\|s(x_{k})\|^{2}\\
&\geq & c_{1}\rho\kappa_{1}\epsilon^{\left(\frac{1-\theta_{\min}}{\theta_{\min}}\right)}\epsilon^{2}\\
&=    &\kappa_{2}\epsilon^{\left(1+\frac{1}{\theta_{\min}}\right)},
\end{eqnarray*}
for $k=0,\ldots,T-1$. Now, summing up these inqualities and using A2, we get
\begin{equation*}
\sum_{k=0}^{T-1}\left[\nu_{k}\right]_{i}+f_{i}(x_{0})-f_{i}^{*}\geq T\kappa_{2}\epsilon^{\left(1+\frac{1}{\theta_{\min}}\right)},
\end{equation*}
which gives
\begin{equation}
\dfrac{1}{\kappa_{2}T}\sum_{k=0}^{T-1}\left[\nu_{k}\right]_{i}+\dfrac{f_{i}(x_{0})-f_{i}^{*}}{\kappa_{2}T}\geq\epsilon^{\left(1+\frac{1}{\theta_{\min}}\right)}.
\label{eq:3.18}
\end{equation}
Suppose that
\begin{equation*}
T\geq C_{i}\left(\dfrac{\kappa_{2}}{2}\epsilon^{\left(1+\frac{1}{\theta_{\min}}\right)}\right).
\end{equation*}
In view of the definition of $C_{i}(\,\cdot\,)$ in (\ref{eq:3.14}), this means that 
\begin{equation}
\dfrac{1}{T}\sum_{k=0}^{T-1}\left[\nu_{k}\right]_{i}\leq\dfrac{\kappa_{2}}{2}\epsilon^{\left(1+\frac{1}{\theta_{\min}}\right)}.
\label{eq:3.19}
\end{equation}
In this case, combining (\ref{eq:3.18}) and (\ref{eq:3.19}), it follows that
\begin{equation*}
\dfrac{f_{i}(x_{0})-f_{i}^{*}}{\kappa_{2}T}\geq \dfrac{1}{2}\epsilon^{\left(1+\frac{1}{\theta_{\min}}\right)}
\end{equation*}
and so
\begin{equation*}
T\leq\dfrac{2(f_{i}(x_{0})-f_{i}^{*})}{\kappa_{2}}\epsilon^{-\left(1+\frac{1}{\theta_{\min}}\right)}
\end{equation*}
Therefore, in any case we have
\begin{equation}
T\leq\max\left\{C_{i}\left(\dfrac{\kappa_{2}}{2}\epsilon^{\left(1+\frac{1}{\theta_{\min}}\right)}\right),\dfrac{2(f_{i}(x_{0})-f_{i}^{*})}{\kappa_{2}}\epsilon^{-\left(1+\frac{1}{\theta_{\min}}\right)}\right\}.
\label{eq:3.20}
\end{equation}
Since $i\in I$ was arbitrarily chose, it follows that (\ref{eq:3.20}) holds for all $i\in I$. Consequently, (\ref{eq:3.16}) is true.
\end{proof}
As a consequence of Theorem \ref{thm:3.1} we have the following global convergence result for  Algorithm~\ref{algo1}.
\vspace{0.2cm}
\begin{corollary}
\label{cor:3.1}
Suppose that A1-A3 hold and let $\left\{x_{k}\right\}_{k\geq 0}$ be a sequence generated by Algorithm~\ref{algo1}. Then, either exists $\bar{k}$ such that $s(x_{\bar{k}})=0$ or 
\begin{equation}
\liminf_{k\to +\infty}\|s(x_{k})\|=0.
\label{eq:3.21}
\end{equation}
\end{corollary}

\begin{proof}
Let $\epsilon\in (0,1)$. From Theorem \ref{thm:3.1}, if 
\begin{equation*}
T>\min_{i\in I}\max\left\{C_{i}\left(\dfrac{\kappa_{2}}{2}\epsilon^{\left(1+\frac{1}{\theta_{\min}}\right)}\right),\dfrac{2(f_{i}(x_{0})-f_{i}^{*})}{\kappa_{2}}\epsilon^{-\left(1+\frac{1}{\theta_{\min}}\right)}\right\}
\end{equation*}
then
\begin{equation*}
\min_{k=0,\ldots,T-1}\|s(x_{k})\|\leq\epsilon.
\end{equation*}
Since $\epsilon\in (0,1)$ was chosen arbitrarily, this shows that 
\begin{equation*}
\lim_{k\to +\infty}\left(\min_{k=0,\ldots,T-1}\|s(x_{k})\|\right)=0.
\end{equation*}
Thus, either there exists $\bar{k}$ such that $s(x_{\bar{k}})=0$ or (\ref{eq:3.21}) holds.
\end{proof}

\section{Particular Cases}\label{sec3}

In Algorithm~\ref{algo1}, different choices for $\nu_{k,\ell}\in\mathbb{R}^{m}_{+}$ in Step 2.2 produce different methods. For example, if
\begin{equation}
\text{$\nu_{k,\ell}=\nu_{k}\equiv 0$ for all $k$ and $\ell$},
\label{eq:4.1}
\end{equation}
then Algorithm~\ref{algo1} reduces to an instance of the monotone projected gradient method proposed in \cite{Drummond}. Clearly, this choice gives a sequence $\left\{\nu_{k}\right\}_{k\geq 0}$ that satisfies A3 and for which (\ref{eq:3.14}) gives 
\begin{equation*}
C_{i}(\delta)=1\quad\text{for each $i\in I$}. 
\end{equation*}
Therefore, it follows from Theorem \ref{thm:3.1} that the monotone version of Algorithm~\ref{algo1} takes at most $\mathcal{O}\left(\epsilon^{-\left(1+\frac{1}{\theta_{\min}}\right)}\right)$ iterations to generate a $\epsilon$-approximate Pareto critical point of (\ref{eq:2.1})-(\ref{eq:2.2}).

Let us consider now the choice
\begin{equation}
\nu_{k,\ell}=\nu_{k}\equiv\left\{\begin{array}{ll}0,&\text{if}\,\,k=0,\\ 
(1-\delta_{k})\left(F(x_{k-1})+\nu_{k-1}-F(x_{k})\right),&\text{if}\,\,k\geq 1,\end{array}\right.\quad\text{for all $\ell$}
\label{eq:4.2}
\end{equation}
with
\begin{equation}
\text{$\delta_{k}\in [\delta_{\min},1]$ for all $k$, and $\delta_{\min}\in (0,1)$.}
\label{eq:4.3}
\end{equation}
For the choice specified by (\ref{eq:4.2}) and (\ref{eq:4.3}), the next lemma establishes that $\left\{[\nu_{k}]_{i}\right\}_{k\geq 0}$ is summable for all $i\in I$. The proof is an adaptation of the proof of Lemma Theorem 4 in \cite{GeovaniSachs1}.
\vspace{0.2cm}
\begin{lemma}
\label{lem:4.1}
Suppose that A1-A3 hold and let $\left\{x_{k}\right\}_{k\geq 0}$ be generated by  Algorithm~\ref{algo1}. If $\left\{\nu_{k}\right\}_{k\geq 0}$ is defined by (\ref{eq:4.2}) and (\ref{eq:4.3}), then $\nu_{k}\in\mathbb{R}^{m}_{+}$ for all $k$ and 
\begin{equation}
\sum_{k=0}^{+\infty}\left[\nu_{k}\right]_{i}<+\infty\quad\text{for all $i\in I$}.
\label{eq:4.4}
\end{equation}
\end{lemma}

\begin{proof}
By (\ref{eq:4.2}), $\nu_{0}=0\in\mathbb{R}^{m}_{+}$. In view of (\ref{eq:3.3}), for all $i\in I$ we have
\begin{equation*}
    f_{i}(x_{k})+[\nu_{k}]_{i}-f_{i}(x_{k+1})\geq\rho\alpha_{k}\left(-\nabla f_{i}(x_{k})^{T}d(x_{k})\right).
\end{equation*}
Consequently, it follows from (\ref{eq:4.3}) that
\begin{equation*}
    (1-\delta_{k+1})\left(f_{i}(x_{k})+[\nu_{k}]_{i}-f_{i}(x_{k+1})\right)\geq (1-\delta_{k+1})\rho\alpha_{k}\left(-\nabla f_{i}(x_{k})^{T}d(x_{k})\right).
\end{equation*}
Then, by (\ref{eq:4.2}) and (\ref{eq:3.13}) we get
\begin{eqnarray}
\left[\nu_{k+1}\right]_{i}&=& (1-\delta_{k+1})\left(f_{i}(x_{k})+\left[\nu_{k}\right]_{i}-f_{i}(x_{k+1})\right)\label{eq:extra4}\\
&\geq & (1-\delta_{k+1})\rho\alpha_{k}\left(-\nabla f_{i}(x_{k})^{T}d(x_{k})\right)\nonumber\\
&\geq & c_{1}(1-\delta_{k+1})\rho\alpha_{k}\|s(x_{k})\|^{2}.\nonumber
\end{eqnarray}
Therefore, $\nu_{k+1}\in\mathbb{R}^{m}_{+}$. Now, let $N\geq 1$. Combining (\ref{eq:4.3}) and (\ref{eq:extra4}) we get
\begin{eqnarray*}
\sum_{k=0}^{N}\left[\nu_{k}\right]_{i}&=&\sum_{k=0}^{N-1}\left[\nu_{k+1}\right]_{i} =  \sum_{k=0}^{N-1}(1-\delta_{k+1})\left(f_{i}(x_{k})+\left[\nu_{k}\right]_{i}-f_{i}(x_{k+1})\right)\\
                                                         & \leq &\sum_{k=0}^{N-1}\left(1-\delta_{\min}\right)(f_{i}(x_{k})-f_{i}(x_{k+1})+[\nu_{k}]_{i})\\
                                                         & =    & \left(1-\delta_{\min}\right)(f_{i}(x_{0})-f_{i}(x_{N}))+\sum_{k=0}^{N-1}[\nu_{k}]_{i}-\delta_{\min}\sum_{k=0}^{N-1}[\nu_{k}]_{i}\\
&\leq & \left(1-\delta_{\min}\right)(f_{i}(x_{0})-f_{i}^{*})+\sum_{k=0}^{N}[\nu_{k}]_{i}-\delta_{\min}\sum_{k=0}^{N-1}[\nu_{k}]_{i},
\end{eqnarray*}
where the last inequality is due to A2 and thet fact that $[\nu_{N}]_{i}\geq 0$.
Thus
\begin{equation*}
\delta_{\min}\sum_{k=0}^{N-1}[\nu_{k}]_{i}\leq (1-\delta_{\min})(f_{i}(x_{0})-f_{i}^{*})
\end{equation*}
and so
\begin{equation*}
\sum_{k=0}^{N-1}[\nu_{k}]_{i}\leq\left(\dfrac{1-\delta_{\min}}{\delta_{\min}}\right)(f_{i}(x_{0})-f_{i}^{*}).
\end{equation*}
Since $i\in I$ and $N\geq 1$ were taken arbitrarily, this means that (\ref{eq:4.4}) is true.
\end{proof}

In view of Lemma \ref{lem:4.1}, the sequence $\left\{\nu_{k}\right\}$ defined by (\ref{eq:4.2}) and (\ref{eq:4.3}) satisfies A3. Moreover, for this sequence, (\ref{eq:3.14}) holds with
\begin{equation*}
C_{i}(\delta)=\left(\dfrac{1-\delta_{\min}}{\delta_{\min}}\right)(f_{i}(x_{0})-f_{i}^{*})\delta^{-1}.
\end{equation*} 
Therefore, it follows from Theorem \ref{thm:3.1} that the corresponding instance of Algorithm~\ref{algo1} takes at most  $\mathcal{O}\left(\epsilon^{-\left(1+\frac{1}{\theta_{\min}}\right)}\right)$ iterations to generate a $\epsilon$-approximate Pareto critical point of (\ref{eq:2.1})-(\ref{eq:2.2}).

Notice that the sequence $\left\{\nu_{k}\right\}_{k\geq 0}$ given by (\ref{eq:4.2}) and (\ref{eq:4.3}) is implicitly defined by the choice of the sequence $\left\{\delta_{k}\right\}\subset [\delta_{\min},1]$. One possibility is to use 
\begin{equation}
\delta_{k}=\dfrac{1}{Q_{k}},\quad\text{for all $k\geq 1$},
\label{eq:4.6}
\end{equation}
where 
\begin{equation}
Q_{0}=1\quad\text{and}\quad Q_{k+1}=\eta_{k}Q_{k}+1
\label{eq:4.7}
\end{equation}
for a given sequence $\eta_{k}\in [\eta_{\min},\eta_{\max}]$ with $0\leq\eta_{\min}\leq\eta_{\max}<1$. In this case we have
\begin{equation*}
Q_{k+1}=1+\sum_{j=0}^{k}\Pi_{i=0}^{j}\eta_{k-j}\leq\sum_{j=0}^{+\infty}\eta_{\max}^{j}=\dfrac{1}{1-\eta_{\max}}
\end{equation*}
and so
\begin{equation*}
\delta_{k+1}\geq 1-\eta_{\max}\equiv \delta_{\min}.
\end{equation*}
By (\ref{eq:4.2}), we also have
\begin{equation}
F(x_{k+1})+\nu_{k+1}=(1-\delta_{k+1})(F(x_{k})+\nu_{k})+\delta_{k+1}F(x_{k+1}).
\label{eq:4.8}
\end{equation}
Thus, denoting
\begin{equation*}
C_{k}=F(x_{k})+\nu_{k},
\end{equation*}
it follows from (\ref{eq:4.8}), (\ref{eq:4.6}) and (\ref{eq:4.7}) that
\begin{equation*}
C_{k+1}=\dfrac{\eta_{k}Q_{k}C_{k}+F(x_{k+1})}{Q_{k+1}}.
\end{equation*}
Moreover, using this notation, condition (\ref{eq:3.3}) can be rewritten as 
\begin{equation*}
f_{i}(x_{k}+\beta^{\ell}d(x_{k}))\leq [C_{k}]_{i}+\rho\beta^{\ell}\nabla f_{i}(x_{k})^{T}d(x_{k})\quad\text{for all $i\in I$}.
\end{equation*}
This means that Algorithm~\ref{algo1} with $\left\{\nu_{k}\right\}$ given by (\ref{eq:4.2}), (\ref{eq:4.6}) and (\ref{eq:4.7}) reduces to:
\begin{itemize}
\item an instance of the nonmonotone projected gradient method proposed in \cite{FS}, when $m_{k}=0$ for all $k\geq 0$; and
\item the variant of Algorithm 6 proposed in \cite{MITA} for the case $\Omega=\mathbb{R}^{n}$, with nonmonotone term inspired by \cite{ZH}.
\end{itemize}

\noindent In particular, it follows from our analysis that these methods also possess an upper complexity bound of  $\mathcal{O}\left(\epsilon^{-\left(1+\frac{1}{\theta_{\min}}\right)}\right)$ iterations. 

Note that multiobjective variants of the nonmonotone line search of \cite{Grippo} can also be seen as particular instances of Algorithm 1 with
\begin{equation*}
\left[\nu_{k,\ell}\right]_{i}=[\nu_{k}]_{i}=\max_{0\leq j\leq M(k)}\,f_{i}(x_{k-j})-f_{i}(x_{k}),\quad \text{for}\quad i=1,\ldots,m,
\end{equation*}
where $M(k)=\min\left\{k,M\right\}$. However, it is not clear whether the corresponding sequence $\left\{\nu_{k}\right\}_{k\geq 0}$ is summable. Thus, the worst-case complexity of this variant remains unknown to us.

\section{A New Nonmonotone Method for Multiobjective Optimization}\label{sec4}

As we saw in the last section, with different choices of $\nu_{k,\ell}$ at Step 2.2 of Algorithm~\ref{algo1}, we obtain different methods. From our analysis, to have a globally convergent method it is enough to select $\nu_{k,\ell}$ such that the corresponding sequence $\left\{\nu_{k}\right\}$ satisfies assumption A3. The next lemma establishes that A3 is very mild, in the sense that it holds for any sequence $\left\{\nu_{k}\right\}$ with $\lim_{k\to +\infty}\nu_{k}=0$.
\vspace{0.2cm}
\begin{lemma}
\label{lem:5.1}
Let $\left\{\nu_{k}\right\}_{k\geq 0}\subset\mathbb{R}^{m}_{+}$ with $\nu_{k}\to 0$. Then $\left\{\nu_{k}\right\}_{k\geq 0}$ satisfies assumption A3.
\end{lemma}

\begin{proof}
Let $i\in I$. By assumption, $\lim_{k\to +\infty}[\nu_{k}]_{i}=0$. Thus, given $\delta>0$, there exists $\xi_{i}(\delta/2)\in\mathbb{N}\setminus\left\{0\right\}$ such that
\begin{equation}
[\nu_{k}]_{i}\leq\dfrac{\delta}{2}\quad\text{for all}\,\,k\geq\xi_{i}(\delta/2).
\label{eq:5.1}
\end{equation}
Moreover, there exists $M_{i}>0$ such that
\begin{equation}
[\nu_{k}]_{i}\leq M_{i}\quad\text{for all}\,\,k\geq 0.
\label{eq:5.2}
\end{equation}
Consider
\begin{equation}
C_{i}(\delta)=\max\left\{\dfrac{2\xi_{i}(\delta/2)M_{i}}{\delta},1+\xi_{i}(\delta/2)\right\}.
\label{eq:5.3}
\end{equation}
If $T\geq C_{i}(\delta)$ we have
\begin{eqnarray*}
\dfrac{1}{T}\sum_{k=0}^{T-1}[\nu_{k}]_{i}&=& \dfrac{1}{T}\left(\sum_{k=0}^{\xi_{i}(\delta/2)-1}[\nu_{k}]_{i}\right)+\dfrac{1}{T}\left(\sum_{k=\xi_{i}(\delta/2)}^{T-1}[\nu_{k}]_{i}\right)\\
&\leq &  \dfrac{1}{T}\left(\sum_{k=0}^{\xi_{i}(\delta/2)-1}M_{i}\right)+\dfrac{1}{T}\left(\sum_{k=\xi_{i}(\delta/2)}^{T-1}\dfrac{\delta}{2}\right)\\
&\leq &  \dfrac{1}{T}\xi_{i}(\delta/2)M_{i}+\dfrac{1}{T}\left(\sum_{k=0}^{T-1}\dfrac{\delta}{2}\right)\\
& = & \dfrac{\delta}{2}+\dfrac{\delta}{2}\\
& = & \delta.
\end{eqnarray*}
This shows that
\begin{equation*}
\lim_{T\to +\infty}\dfrac{1}{T}\sum_{k=0}^{T-1}[\nu_{k}]_{i}=0.
\end{equation*}
Since $i\in I$ was taken arbitrarily, we conclude that $\left\{\nu_{k}\right\}_{k\geq 0}$ satisfies A3.
\end{proof}

It follows from Lemma \ref{lem:5.1} and Corollary \ref{cor:3.1} that any instance of Algorithm~\ref{algo1} with $\nu_{k}\to 0$ is globally convergent for the problem class specified by assumptions A1 and A2. This gives a great deal of freedom for the development of new nonmonotone methods for multiobjective optimization problems. As an example, let us consider the following new instance of Algorithm~\ref{algo1}.
\newpage
\begin{algorithm}[ht!]
\caption{Metropolis-Based Nonmonotone Line Search Method}\label{algo2}
\noindent\textbf{Step 0.} Choose $x_{0}\in\Omega$, $\beta,\rho\in (0,1)$, $c_{1},c_{2}>0$, $\sigma\in\mathbb{R}^{m}_{+}$, $\gamma>0$, and $\left\{\tau_{k}\right\}_{k\geq 0}\subset\mathbb{R}_{++}$ with $\tau_{k}\to 0$. Set $k:=0$.
\\[0.2cm]
\noindent\textbf{Step 1.} Compute $d(x_{k})\in\Omega-x_{k}$ such that
\begin{equation*}
\max_{i\in I}\left\{\nabla f_{i}(x_{k})^{T}d(x_{k})\right\}\leq -c_{1}\|s(x_{k})\|^{2}
\end{equation*} 
and
\begin{equation*}
\|d(x_{k})\|\leq c_{2}\|s(x_{k})\|.
\end{equation*}
\\[0.2cm]
\noindent\textbf{Step 2.1.} Set $\ell:=0$.
\\[0.2cm]
\noindent\textbf{Step 2.2.} Compute $x_{k,\ell}^{+}=x_{k}+\beta^{\ell}d(x_{k})$ and define
\begin{equation}
[\nu_{k,\ell}]_{i}=\sigma_{i}\text{exp}\left(-\dfrac{\max\left\{\gamma,f_{i}(x_{k,\ell}^{+})-f_{i}(x_{k})\right\}}{\tau_{k}}\right),\quad\forall i\in I.
\label{eq:5.4}
\end{equation}
If 
\begin{equation*}
f_{i}(x_{k,\ell}^{+})\leq f_{i}(x_{k})+\rho\beta^{\ell}\nabla f_{i}(x_{k})^{T}d(x_{k})+\left[\nu_{k,\ell}\right]_{i}\quad\forall i\in I,
\end{equation*}
set $\ell_{k}=\ell$ and go to Step 3. Otherwise, set $\ell:=\ell+1$ and repeat Step 2.2.
\\[0.2cm]
\noindent\textbf{Step 3.} Set $\nu_{k}=\nu_{k,\ell_{k}}$, $\alpha_{k}=\beta^{\ell_{k}}$, $x_{k+1}=x_{k,\ell_{k}}^{+}$, $k:=k+1$ and go back to Step 1.
\end{algorithm}

Algorithm~\ref{algo2} is a generalization of the Metropolis-based nonmonotone method proposed in \cite{GeovaniSachs} for single-objective optimization. As a consequence of Corollary \ref{cor:3.1} we have the following global convergence result for Algorithm~\ref{algo2}. 
\vspace{0.2cm}
\begin{theorem}
\label{thm:5.1}
Suppose that A1-A2 hold, and let $\left\{x_{k}\right\}_{k\geq 0}$ be a sequence generated by Algorithm~\ref{algo2}. Then, either there exists $\bar{k}$ such that $d(x_{\bar{k}})=0$ or
\begin{equation*}
\liminf_{k\to +\infty}\|s(x_{k})\|=0.
\end{equation*}
\end{theorem}

\begin{proof}
In view of Corollary \ref{cor:3.1} and Lemma \ref{lem:5.1} it is enough to show that $\left\{\nu_{k}\right\}_{k\geq 0}$ generated by Algorithm~\ref{algo2} converges to $0\in\mathbb{R}^{m}$. By Steps 2.2 and 3 in Algorithm~\ref{algo2}, for all $i\in I$ and $k\geq 0$, we have
\begin{equation*}
0\leq [\nu_{k}]_{i}=\sigma_{i}\text{exp}\left(-\dfrac{\max\left\{\gamma,f_{i}(x_{k+1})-f_{i}(x_{k})\right\}}{\tau_{k}}\right)\leq\sigma_{i}\text{exp}\left(-\frac{\gamma}{\tau_{k}}\right).
\end{equation*}
Since $\gamma,\tau_{k}>0$ and $\tau_{k}\to 0$, it follows that
\begin{equation*}
\lim_{k\to +\infty}[\nu_{k}]_{i}=0,\quad\forall i\in I,
\end{equation*}
and so $\nu_{k}\to 0$.
\end{proof}

From the proof Lemma \ref{lem:5.1} we see that if $\nu_{k}\to 0$ then (\ref{eq:3.14}) holds for
\begin{equation*}
C_{i}(\delta)=\max\left\{\dfrac{2\xi_{i}(\delta/2)M_{i}}{\delta},1+\xi_{i}(\delta/2)\right\},
\end{equation*} 
where $M_{i}$ is a uniform upper bound to $\left\{[\nu_{k}]_{i}\right\}_{k\geq 0}$, and $\xi_{i}(\delta/2)$ is any positive integer such that
\begin{equation*}
[\nu_{k}]_{i}\leq\dfrac{\delta}{2},\quad\forall k\geq \xi_{i}(\delta/2).
\end{equation*}
Thus, when $\nu_{k}\to 0$, it follows from Theorem \ref{thm:3.1} that Algorithm 1 takes no more than
\begin{equation}
\left\lceil \min_{i\in I}\max\left\{\dfrac{4\xi_{i}\left(\frac{\kappa_{2}}{4}\epsilon^{\left(1+\frac{1}{\theta_{\min}}\right)}\right)M_{i}}{\kappa_{2}\epsilon^{\left(1+\frac{1}{\theta_{min}}\right)}},1+\xi_{i}\left(\frac{\kappa_{2}}{4}\epsilon^{\left(1+\frac{1}{\theta_{\min}}\right)}\right),\frac{2(f_{i}(x_{0})-f_{i}^{*})}{\kappa_{2}\epsilon^{\left(1+\frac{1}{\theta_{\min}}\right)}}\right\}\right\rceil
\label{eq:bound}
\end{equation}
iterations to find a $\epsilon$-approximate Pareto critical point of (\ref{eq:2.1})-(\ref{eq:2.2}). Therefore, to obtain an explicit iteration complexity bound, all that we need to do is to estimate the rate of decay of $\left\{[\nu_{k}]_{i}\right\}_{k\geq 0}$, which will allow the identification of $M_{i}$ and $\xi_{i}(\delta/2)$. From the proof of Theorem \ref{thm:5.1}, we know that the sequence $\left\{[\nu_{k}]_{i}\right\}_{k\geq 0}$ in Algorithm 2 satisfies
\begin{equation*}
0\leq [\nu_{k}]_{i}\leq\sigma_{i}\text{exp}\left(-\frac{\gamma}{\tau_{k}}\right),\quad\forall k\geq 0.
\end{equation*}
Let us consider the choice $\tau_{k}=1/\ln(k+1)$. In this case, we get
\begin{equation}
[\nu_{k}]_{i}\leq\sigma_{i}\text{exp}(-\gamma\ln(k+1))=\sigma_{i}\text{exp}\left(\ln\left((k+1)^{-\gamma}\right)\right)=\dfrac{\sigma_{i}}{(k+1)^{\gamma}}.
\label{eq:decay}
\end{equation}
This implies that
\begin{equation}
[\nu_{k}]_{i}\leq\sigma_{i}\equiv M_{i},\quad\forall k\geq 0,
\label{eq:cond1}
\end{equation}
and
\begin{equation}
[\nu_{k}]_{i}\leq\dfrac{\delta}{2},\quad\forall k\geq\left(\frac{2\sigma_{i}}{\delta}\right)^{\frac{1}{\gamma}}\equiv\xi_{i}(\delta/2).
\label{eq:cond2}
\end{equation}
Therefore, it follows from (\ref{eq:bound}), (\ref{eq:cond1}) and (\ref{eq:cond2}) that Algorithm 2 with $\tau_{k}=1/\ln(k+1)$ takes no more than $\mathcal{O}\left(\epsilon^{-\left(1+\frac{1}{\theta_{\min}}\right)\left(1+\frac{1}{\gamma}\right)}\right)$ iterations to find a $\epsilon$-approximate Pareto critical point of (\ref{eq:2.1})-(\ref{eq:2.2}). 

For the case $\gamma>1$, an improved complexity bound can be obtained for Algorithm 2. Indeed, if $\gamma>1$, then it follows from (\ref{eq:decay}) that 
\begin{equation*}
\sum_{k=0}^{+\infty}[\nu_{k}]_{i}\leq\sigma_{i}\sum_{k=0}^{+\infty}\dfrac{1}{(k+1)^{\gamma}}\leq\dfrac{\sigma_{i}\gamma}{\gamma-1}.
\end{equation*}
In this case, (\ref{eq:3.14}) is satisfied with $C_{i}(\delta)=\left(\frac{\sigma_{i}\gamma}{\gamma-1}\right)\delta^{-1}$. Consequently, by Theorem \ref{thm:3.1}, Algorithm 2 with $\gamma>1$ and $\tau_{k}=1/\ln(k+1)$ takes at most $\mathcal{O}\left(\epsilon^{-\left(1+\frac{1}{\theta_{\min}}\right)}\right)$ iterations to find a $\epsilon$-approximate Pareto critical point.

\section{Illustrative Numerical Results}\label{sec5}

In \cite{GeovaniSachs}, a variant of Algorithm 2 for single-objective optimization (case $m=1$) showed a remarkable ability to escape non-global local minimizers. Here, we investigate the performance of Algorithm 2 applied to bi-objective optimization problems where one of the objectives has numerous non-global local minimizers. Specifically, we considered  $15$ bi-objective problems of the form
\begin{equation*}
\begin{array}{rc} \min_{x\in\mathbb{R}^{n}}& F(x)=(f_{1}(x),f_{2}(x))\\
                         \text{s.t.} & x\in [-a,a]^{n},
\end{array}
\end{equation*}
with $a=5.12$, 
\begin{equation*}
f_{1}(x)=10n+\sum_{i=1}^{n}\left[x_{1}^{2}-10\cos(2\pi x_{i})\right],\quad\forall x\in\mathbb{R}^{n},\quad\forall x\in\mathbb{R}^{n},
\end{equation*}
and $f_{2}(\,\cdot\,)$ being one of the $15$ functions from the MGH collection \cite{MORE} whose dimension $n$ can be chosen\footnote{Namely, \texttt{Extended Rosenbrock}, \texttt{Extended Powell Singular}, \texttt{Penalty I}, \texttt{Penalty II}, \texttt{Variably Dimensioned}, \texttt{Trigonometric}, \texttt{Discrete Boundary Value}, \texttt{Discrete Integral Equation}, \texttt{Broyden Tridiagonal}, \texttt{Broyden Banded}, \texttt{Brown Almost Linear}, \texttt{Linear}, \texttt{Linear-1}, \texttt{Linear-0}, \texttt{Chebyquad}.}. Function $f_{1}(\,\cdot\,)$ is known as the \textit{Rastrigin function}. This function has a large number of spurious minimizers in the hypercube $[-5.12,5.12]^{n}$ (see, e.g., \cite{RUD}). The following \texttt{Julia} codes were compared:
\begin{itemize}
\item \textbf{M}: the monotone version of Algorithm 1, obtained with $\nu_{k,\ell}=0$ for all $k$ and $\ell$, and $m_{k}=0$ for all $k$.
\item \textbf{N1}: Algorithm 1 with $\nu_{k,\ell}$ given by (\ref{eq:4.2}), (\ref{eq:4.6}) and (\ref{eq:4.7}), $\eta_{k}=0.85/(k+1)$, and $m_{k}=0$ for all $k$.
\item \textbf{N2}: Algorithm 2 with $\tau_{k}=1/\ln(k+1)$, $\gamma=8$, and $\sigma_{i}=|f_{i}(x_{0})|$ for all $i\in I$.
\item \textbf{Nh:} The hybrid-type nonmonotone line search as proposed in \cite{MITA}, which uses $m_{k}=\lceil m/2\rceil$ for all $k$.
\end{itemize}
In all implementations, we considered $d(x_{k})=s(x_{k})$ (for which $c_{1}=c_{2}=1$), and the parameters $\rho=10^{-4}$ and $\beta=0.5$. Focusing on the case $n=4$, for each problem we tested $81$ choices for the starting point $x_{0}\in\mathbb{R}^{4}$, namely, 
\begin{equation*}
x_{0}=\left[-2a+ia,-2a+ja,-2a+ka,-2a+\ell a\right]^{T},\quad i,j,k,\ell \in \left\{1,2,3\right\}.
\end{equation*}
This resulted in a total of $1215$ pairs $(\text{problem},\text{starting point})$. We applied the three solvers in all these pairs with stopping criterion
\begin{equation}\label{cp1}
\|s(x_{k})\|\leq\epsilon\equiv 10^{-4},
\end{equation}
allowing a maximum of $1000$ iterations for each solver. 
All the experiments
were performed with \texttt{Julia} $1.7.2$ on a PC with Intel(R) Core(TM) i7-10510U  with  microprocessor 1.8 GHz and 32 GB RAM. We use \texttt{Gurobi} and \texttt{JuMP} \cite{DunningHuchetteLubin2017} to compute $s(x_{k})$.

Codes are compared using performance profile  \cite{profile}. Given a set of solvers $\mathcal{S}$, and a set of test problems $\mathcal{P}$,  denote by $t_{p,s}>0$ the performance of the solver $s \in\mathcal{S}$ applied to problem $ p \in \mathcal{P}$. Then the performance profile of solver $s$ is the graph of the function   $\gamma_s:[1, \infty) \to [0,1]$ given by
$$
\gamma_s(\tau) = \dfrac{1}{\vert P \vert} \Big|\Big\{p \in\mathcal{P}: \dfrac{t_{p,s}}{\min\{t_{p,s}:s \in \mathcal{S}\}} \leq \tau \Big\}\Big|.
$$
Note that, $\gamma_s(1)$ is the percentage of problems for which solver $s$ wins over the rest of the solvers. 
\subsection{Number of Iterations to achieve $\epsilon$-Pareto Criticality}

Figure \ref{fig:iterations} shows the performance profiles of solvers M, N1, N2 and Nh, considering $\mathcal{P}=\left\{\text{(problem, starting point)}\right\}$ and $t_{p,s}$ as the number of iterations that solver $s$ applied to the pair $p\in\mathcal{P}$ requires to generate $x_{k}$ such that (\ref{cp1}) holds. As we can see, N2 was the only nonmonotone method more efficient than the monotone method. In addition, N1 and N2 were more robust than M and Nh.

\begin{figure}[H]
    \centering
    \includegraphics[scale=0.4]{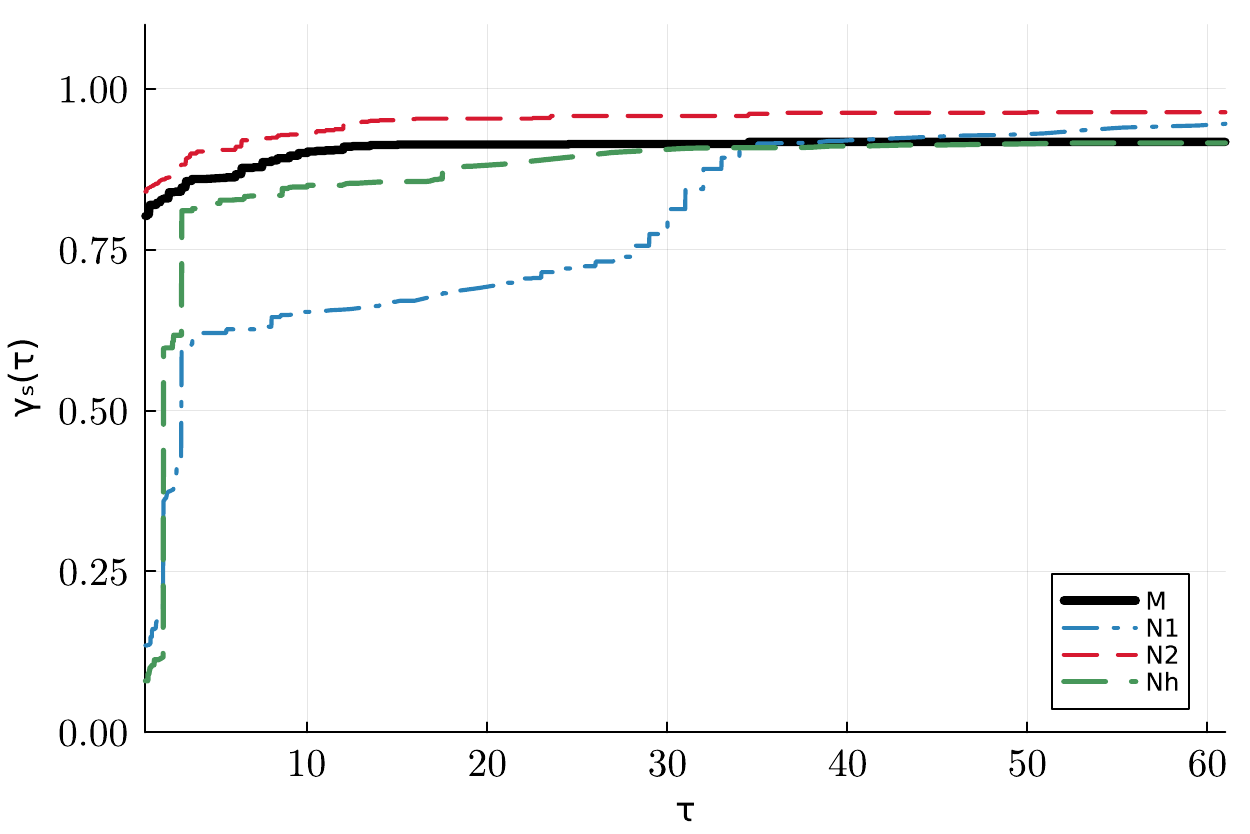}
    \caption{Performance profile with respect to the number of iterations}
    \label{fig:iterations}
\end{figure}

\subsection{Purity Metric}

In the context of multiobjective optimization, other performance measures are also relevant. For example, it is particularly interesting to quantify the ability of the solvers to find Pareto efficient solutions, also known as \textit{Pareto front}. For that, we need first to identify an approximation for this set. By applying solver $s$ to problem $p$ starting from $N$ different initial points, we obtain a set of $N$ approximate Pareto critical points in $\Omega$. Removing from this set all the dominated points\footnote{We say that a point $x_1$ is dominated by a point $x_2$ when $f_i(x_2) \leq f_i(x_1),\; \forall i\in I $   and $f_j(x_2) < f_j(x_1)$ for some $j \in I.$}, we get the set  $PF_{p,s}$. Let $PF_{p}$ be the set obtained by removing all the dominated points from $\cup_{s \in S} PF_{p,s}$. Then, $PF_{p}$ can be seen as an approximation to the Pareto front of problem $p$. The  Purity metric of solver $s$ applied to problem $p$ is defined as 
\begin{equation*}
t_{p,s}=\left\{\begin{array}{ll} 1/\bar{t}_{p,s},&\text{if}\,\,\bar{t}_{p,s}\neq 0,\\
                                              +\infty,&\text{otherwise},
\end{array}
\right.
\end{equation*}
where 
\begin{equation*}
\bar{t}_{p,s}=\dfrac{|PF_{p,s}\cap PF_{p}|}{|PF_{p}|}.
\end{equation*}
When using the Purity metric, as suggested in \cite{MULTISEARCH}, we compare the algorithms in pairs. As we can see in Figure \ref{fig:purity},  our new method N2 outperformed both M, N1 and Nh with respect to the Purity metric.

\newpage

\begin{figure}[h!]
\begin{minipage}[b]{0.75\linewidth}
\centering
    \includegraphics[scale=0.4]{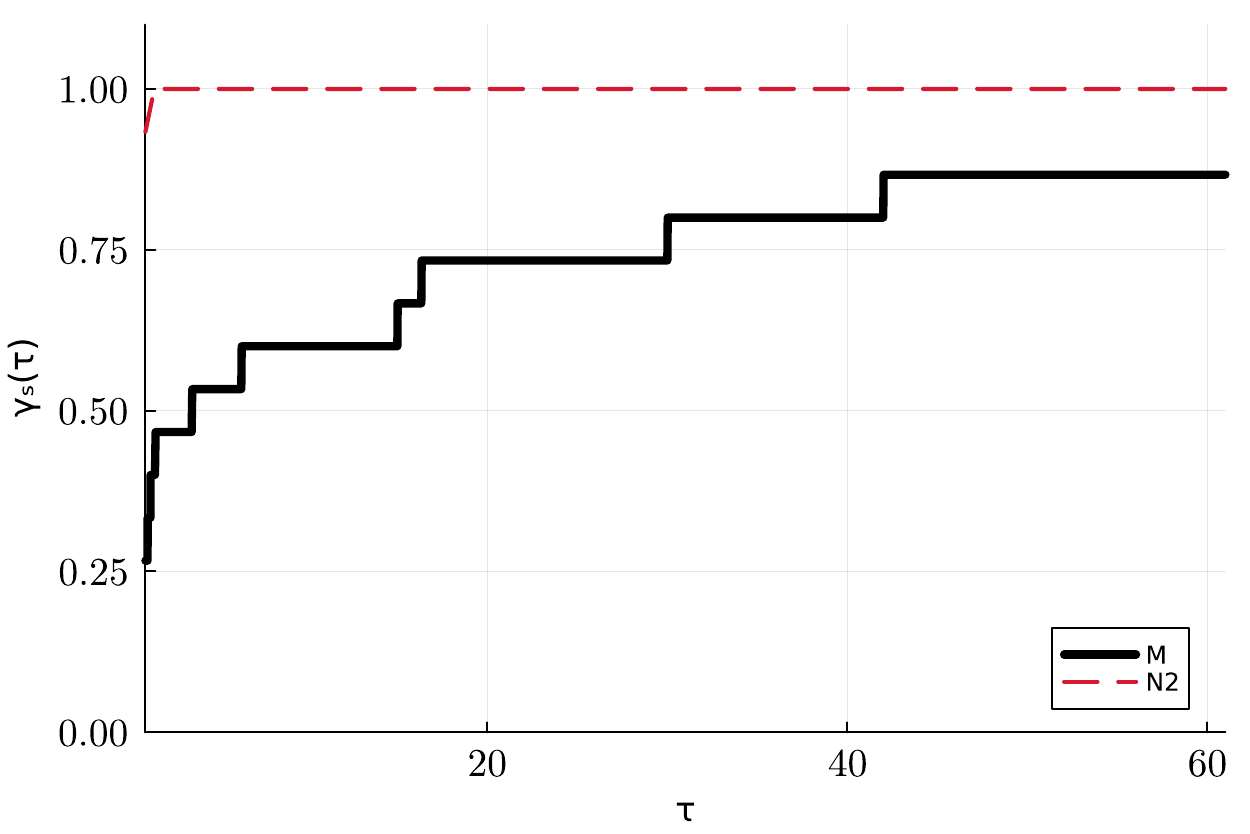}
\end{minipage}
\hfill
\begin{minipage}[b]{0.75\linewidth}
\centering
   \includegraphics[scale=0.4]{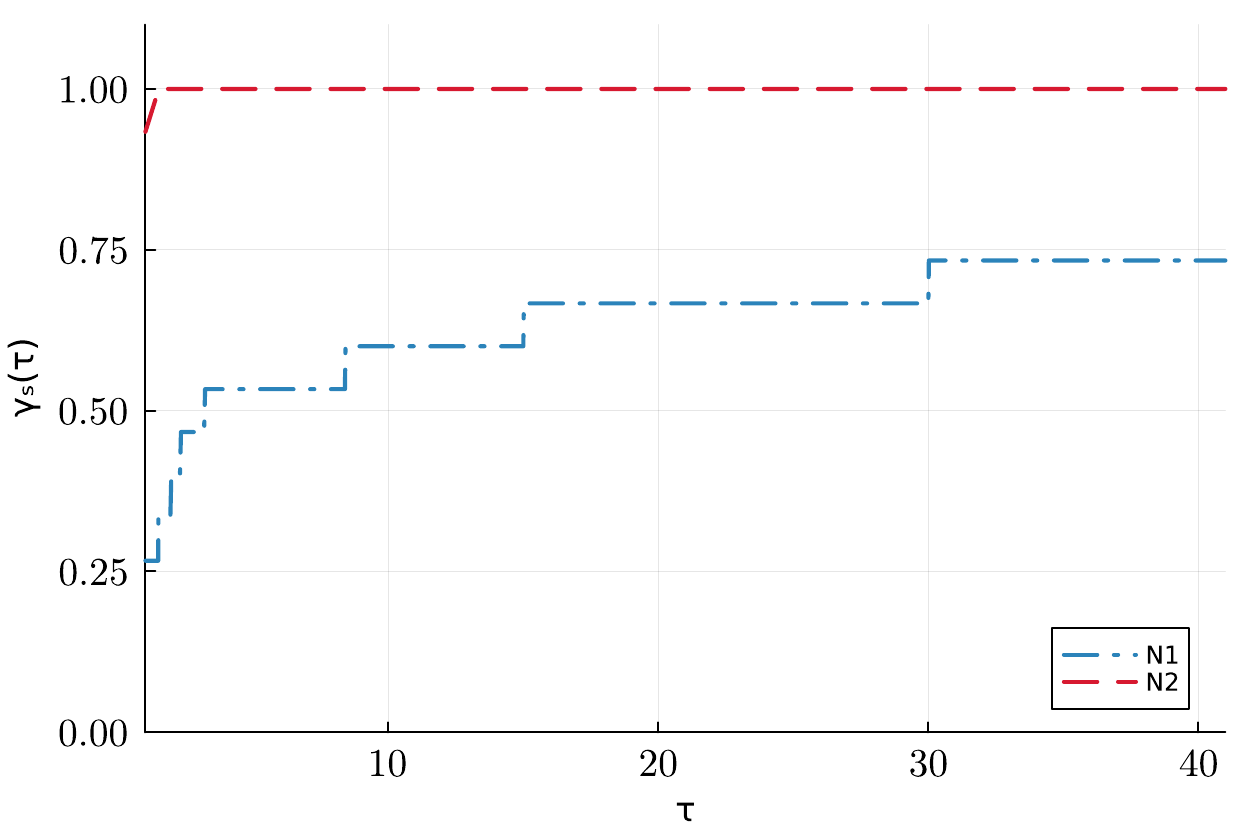}
\end{minipage}
\hfill
\begin{minipage}[b]{0.75\linewidth}
\centering
     \includegraphics[scale=0.4]{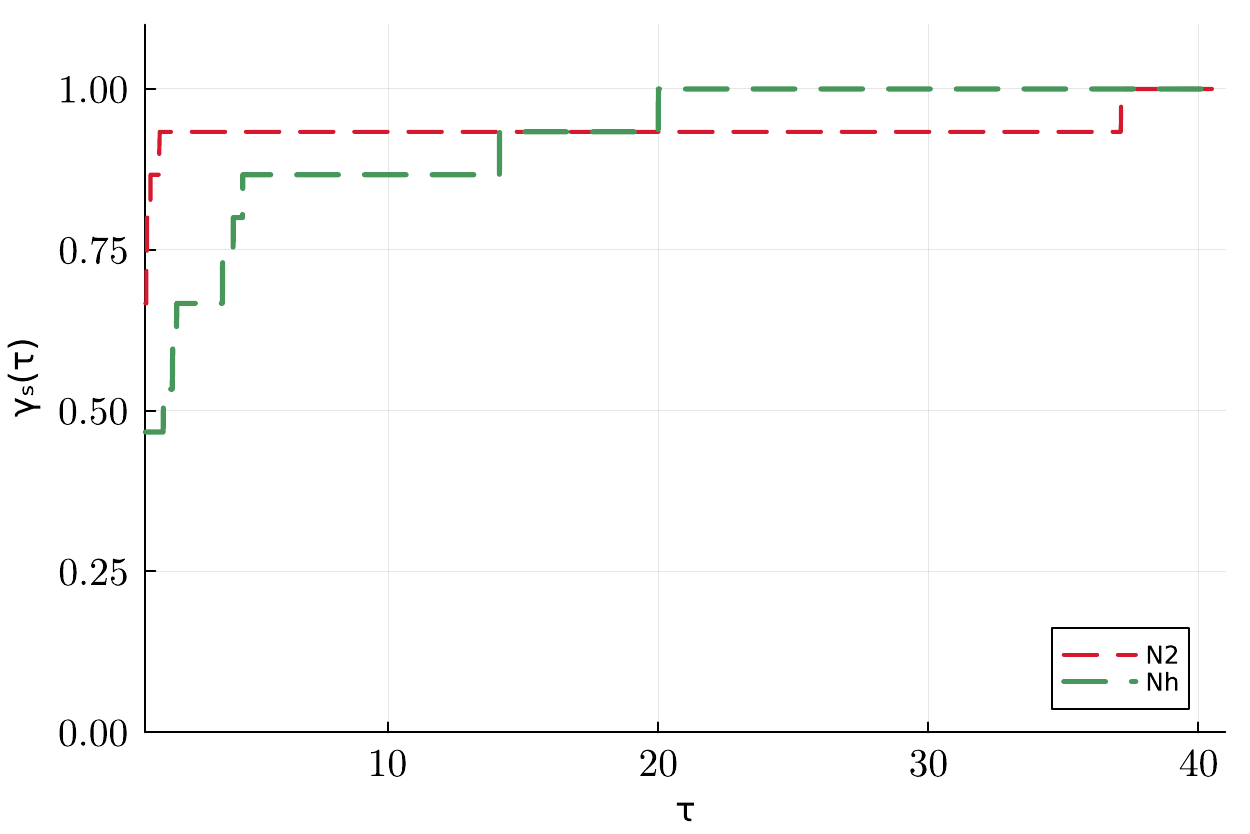}
\end{minipage}
  \caption{Performance profile with respect to the purity metric}
  \label{fig:purity}
\end{figure}

\section{Conclusion}\label{secconc}

In this work, we proposed a universal nonmonotone line search method for nonconvex multiobjective problems with convex constraints. The method is \textit{universal} because it does not require any knowledge about the constants that define the smoothness level of the objectives. Specifically, the step-sizes are selected using a relaxed Armijo condition, allowing the increase of the objectives between consecutive iterations. The degree of nonmonotonicity allowed for the objective $f_{i}(\,\cdot\,)$ is controlled by a nonnegative parameter $\left[\nu_{k}\right]_{i}$. Assuming that $\nabla f_{i}(\,\cdot\,)$ is $\theta_{i}$-H\"{o}lder continuous, and that $\left\{\left[\nu_{k}\right]_{i}\right\}_{k\geq 0}$ is summable for every $i\in\left\{1,\ldots,m\right\}$, we proved that our method takes no more than $\mathcal{O}\left(\epsilon^{-\left(1+\frac{1}{\theta_{\min}}\right)}\right)$ iterations to find a $\epsilon$-approximate critical point of a problem with $m$ objectives and $\theta_{\min}=\min_{i=1,\ldots,m}\left\{\theta_{i}\right\}$. Under the weaker assumption that $\lim_{k\to+\infty}\left[\nu_{k}\right]_{i}=0$ for every $i\in\left\{1,\ldots,m\right\}$, we also proved a liminf-type global convergence result. In particular, we showed that our complexity bound applies to some existing monotone and nonmonotone line search methods. In addition, exploring the generality of our assumptions about $\left\{\nu_{k}\right\}_{k\geq 0}\subset\mathbb{R}_{+}^{m}$, we proposed a new Metropolis-based nonmonotone method that fits into our general scheme. Our preliminary numerical results indicate that this new method performs favorably in comparison to the monotone projected gradient method by Drummond and Iusem \cite{Drummond}, the nonmonotone projected gradient method by Fazzio and Schuverdt \cite{FS}, and a projected variant of the nonmonotone method by Mita, Fukuda, and Yamashita \cite{MITA}, on problems where at least one objective has numerous non-global local minimizers.


\subsection*{Funding}

G.N. Grapiglia was partially supported by the National Council for Scientific and Technological Development (CNPq) - Brazil (grant 312777/2020-5).
M.E. Pinheiro was partially supported by the Coordination for the Improvement of Higher Education Personnel (CAPES) - Brazil.





\bibliographystyle{abbrv}
\bibliography{sn-bibliography}

\end{document}